\documentclass[10pt,a4paper,draft]{article}
\usepackage{pifont}
\usepackage{bbding}
\usepackage{amsmath}
\usepackage{mathrsfs}
\usepackage{amsfonts}
\usepackage{amssymb}
\usepackage{amsfonts,amssymb,amsmath,indentfirst,amsthm}
\usepackage{epsfig}

\numberwithin{equation}{section}
\newenvironment{abs}{\textbf{Abstract}\mbox{  }}{ }
\newenvironment{key words}{\emph{\texttt{Keywords}}\mbox{  }}{ }
\newtheorem{theorem}{Theorem}[section]
\newtheorem{lemma}[theorem]{Lemma}
\newtheorem{corollary}[theorem]{Corollary}
\newtheorem{proposition}[theorem]{Proposition}

\renewenvironment{proof}{\noindent{\textbf{Proof.}}}{\hfill$\Box$}
\newenvironment{altproof}[1]
{\addvspace{0.3cm}\noindent
{\textbf{Proof of Theorem {#1}}.}}
{\nopagebreak\mbox{}\hfill $\Box$\par\addvspace{0.5cm}}
\theoremstyle{remark}
\newtheorem{remark}[theorem]{Remark}

\theoremstyle{plain}

\makeatletter
    \newcommand{\rmnum}[1]{\romannumeral #1}
    \newcommand{\Rmnum}[1]{\expandafter\@slowromancap\romannumeral #1@}
\makeatother

\newcommand{\wt}{\widetilde}
\newcommand{\wh}{\widehat}
\newcommand{\ov}{\overline}

   \def\dist{\mathrm{dist}}
   \def\span{\mathrm{span}}

   \def\R{\mathbb{R}}

\newcommand{\cO}{{\mathcal O}}

\newcommand{\cT}{{\mathcal T}}

\newcommand{\vphi}{\varphi}
\newcommand{\al}{\alpha}
\newcommand{\be}{\beta}

\newcommand{\de}{\delta}
\newcommand{\ze}{\zeta}
\newcommand{\la}{\lambda}
\newcommand{\si}{\sigma}

\newcommand{\De}{\Delta}

\newcommand{\La}{\Lambda}
\newcommand{\Om}{\Omega}

\newcommand{\eps}{\varepsilon}

\newcommand{\weakto}{\rightharpoonup}
\newcommand{\pa}{\partial}

\def\Ker{\mathrm{Ker}}

\begin{document}

\title{\textbf{Nodal bubble tower solutions to slightly subcritical elliptic problems with Hardy terms}}
\author{Thomas Bartsch\thanks{Mathematisches Institut, Justus-Liebig-Universit$\ddot{a}$t Giessen, Arndtstr. 2, 35392 Giessen, Germany; E-mail: thomas.bartsch@math.uni-giessen.de}, Qianqiao Guo\thanks{School of Mathematics and Statistics, Northwestern Polytechnical University, 710129 Xi'an, China; E-mail: gqianqiao@nwpu.edu.cn}
}

\date{}
\maketitle

\begin{center}
  {\large\it Dedicated to the 85th Birthday of Professor Dajun Guo.}
\end{center}
\vspace{.2cm}

\noindent
\begin{abs}
We study the possible blow-up behavior of solutions to the slightly subcritical elliptic problem with Hardy term
\[
\left\{
\begin{aligned}
-\De u-\mu\frac{u}{|x|^2} &= |u|^{2^{\ast}-2-\eps}u &&\quad \text{in } \Om, \\\
u &= 0&&\quad \text{on } \pa\Om,
\end{aligned}
\right.
\]
in a bounded domain $\Om\subset\R^N (N\ge7)$ with $0\in\Om$, as $\mu,\epsilon\to 0^+$. In \cite{BarGuo-ANS}, we obtained the existence of nodal solutions that blow up positively at the origin and negatively at a different point as $\mu=O(\epsilon^\alpha)$ with $\alpha>\frac{N-4}{N-2}$, $\epsilon\to 0^+$. Here we prove the existence of nodal bubble tower solutions, i.e.\ superpositions of bubbles of different signs, all blowing up at the origin but with different blow-up order, as $\mu=O(\epsilon)$, $\epsilon\to0^+$.
\end{abs}

\vspace{.2cm}

\vspace{.2cm}

\section{\textbf{Introduction}\label{Section 1}}
We continue to study the possible blow-up behavior of solutions to the slightly subcritical elliptic problem with Hardy term
\begin{equation}\label{pro}
\left\{
\begin{aligned}
-\De u-\mu\frac{u}{|x|^2} &= |u|^{2^{\ast}-2-\eps}u &&\quad \text{in }\Om, \\\
u &= 0 &&\quad \text{on }\pa \Om,
\end{aligned}
\right.
\end{equation}
where $\Om\subset\R^N$, $N\ge7$, is a smooth bounded domain with $0\in\Om$; $2^*:=\frac{2N}{N-2}$ is the critical Sobolev exponent. 
In \cite{BarGuo-ANS}, for fixed $\alpha>\frac{N-4}{N-2}$ and $\mu_0>0$, we obtain the existence of nodal solutions that blow up positively at the origin and negatively at a different point with $\mu=\mu_0 \epsilon^\alpha, \epsilon\to 0^+$. In this paper we want to study the existence of nodal bubble tower solutions, i.e. superpositions of bubbles of different signs, all blowing up at the origin but with different blow-up order. For that, we need to assume $\alpha=1$, that is, $\mu=\mu_0 \epsilon$.

The blow-up phenomenon for positive and for nodal solutions to problem \eqref{pro} for $\mu=0$ has been studied extensively, see e.g. \cite{BaLiRey95CVPDE,BarMiPis06CVPDE, BarDPis-prepr1,BarDPis-prepr2,BrePe89PNDEA,PiDolMusso03JDE,FluWeip7MM,GroTaka10JFA,Han91AIHPAN, MussoPis10JMPA,PisWe07AIHPA,Rey89MM,Rey90JFA,Rey91DIE} and the references therein. However, it is well known that if $\mu\neq0$, the Hardy potential $\frac{1}{|x|^2}$ cannot be regarded as a lower order perturbation since it has the same homogeneity as the Laplace operator. This makes the analysis interesting and more complicated compared with the case $\mu=0$. The existence of positive and nodal solutions to the problem with Hardy type potentials and critical exponents has been studied in a number of papers, see e.g. \cite{CaoH04JDE,CaoP03JDE,EkeGhou02,FG,GR1,GR2,GY,GuoNiu08JDE,Jan99,RuWill03,Smets05TAMS,Terr96} and the references therein. However, few results are known concerning blow-up solutions. The only results we are aware of are related to the problem
\begin{equation*}
\left\{
\begin{aligned}
&-\De u-\frac{\mu}{|x|^2}u = k(x)u^{2^*-1},\\
&u\in D^{1,2}(\R^N),\ \ u>0 \text{ in } \R^N\setminus\{0\};
\end{aligned}
\right.
\end{equation*}
here $D^{1,2}(\R^N):=\{u\in L^{2^*}(\R^N): |\nabla u|\in L^2(\R^N)\}$; see \cite{FelliPis06CPDE,FelliTer05CCM,MussoWei12IMRN}.


In order to state the main results of this paper, we first introduce some notations. By Hardy's inequality, the norm
\[
\|u\|_\mu:=\left(\int_{\Om}(|\nabla u|^2-\mu\frac{u^2}{|x|^2})dx\right)^{\frac12}
\]
is equivalent to the norm $\|u\|_0=\left(\int_{\Om}|\nabla u|^2dx\right)^{1/2}$ on $H_0^1(\Om)$ provided $0\le\mu<\ov{\mu}:=\frac{(N-2)^2}{4}$. This will of course be the case for $\mu=\mu_0\eps^\al$ with $\eps>0$ small. As in \cite{FG} we write $H_\mu(\Om)$ for the Hilbert space consisting of $H^1_0(\Omega)$ functions with the inner product
\[
(u,v):=\int_{\Om}\left(\nabla u\nabla v-\mu\frac{uv}{|x|^2}\right)dx.
\]
It is known that the nonzero critical points of the energy functional
\[
J_\eps(u)
 := \frac12\int_{\Om}\left(|\nabla u|^2-\mu\frac{u^2}{|x|^2}\right)dx
     - \frac{1}{2^*-\eps}\int_{\Om}|u|^{2^*-\eps}dx
\]
defined on $H_\mu(\Om)$ are precisely the nontrivial weak solutions to problem \eqref{pro}.

We need to recall the ground states of two problems that appear as limiting problems. The ground states of
\begin{equation}\label{limit-pro1}
\left\{
\begin{aligned}
-\De u &= |u|^{2^{\ast}-2}u &&\quad \text{in } \R^N, \\\
u &\to 0 &&\quad \text{as } |x|\to\infty
\end{aligned}
\right.
\end{equation}
are the instantons
\[
U_{\de,\xi} := C_0\left(\frac{\de}{\de^2+|x-\xi|^2}\right)^{\frac{N-2}{2}}
\]
with $\de>0$, $\xi\in \R^N$ and $C_0:=(N(N-2))^{\frac{N-2}{4}}$; see \cite{Aub76,Tal76}. These are the minimizers of 
\[
S_0
 := \min_{u\in D^{1,2}(\R^N)\setminus \{0\}}
         \frac{\int_{\R^N} |\nabla u|^2 dx}{(\int_{\R^N}|u|^{2^*} dx)^{2/{2^*}}}
\]
and they satisfy
$$
\int_{\R^N} |\nabla U_{\de,\xi}|^2dx = \int_{\R^N} |U_{\de,\xi}|^{2^*}dx = S_0^{\frac{N}{2}}.
$$

Secondly, if $0<\mu<\ov{\mu}$ then all positive solutions to
\begin{equation}\label{limit-pro2}
\left\{
\begin{aligned}
-\De u-\mu\frac{u}{|x|^{2}} &= |u|^{2^{\ast}-2}u &&\quad \text{in } \R^N, \\
u &\to 0 &&\quad \text{as } |x|\to\infty
\end{aligned}
\right.
\end{equation}
are given by 
\[
V_\si = C_\mu\left(\frac{\si}{\si^2|x|^{\be_1}+|x|^{\be_2}}\right)^{\frac{N-2}{2}}
\]
with $\si>0$, $\be_1:=(\sqrt{\ov{\mu}}-\sqrt{\ov{\mu}-\mu})/\sqrt{\ov{\mu}}$, $\be_2:=(\sqrt{\ov{\mu}}+\sqrt{\ov{\mu}-\mu})/\sqrt{\ov{\mu}}$, and
$C_\mu:=\left(\frac{4N(\ov{\mu}-\mu)}{N-2}\right)^{\frac{N-2}{4}}$; see \cite{CatW00,Terr96}.
These solutions minimize
\[
S_\mu
 := \min_{u\in D^{1,2}(\R^N)\setminus \{0\}}
     \frac{\int_{\R^N}(|\nabla u|^2-\mu\frac{u^2}{|x|^2})dx}
          {(\int_{\R^N}|u|^{2^*}dx)^{2/{2^*}}}
\]
and there holds
$$
\int_{\R^N} \left(|\nabla V_\si|^2-\mu\frac{|V_\si|^2}{|x|^2}\right)dx
 = \int_{\R^N} |V_\si|^{2^*}dx = S_\mu^{\frac{N}{2}}.
$$

The Green's function of the Dirichlet Laplacian is given by
$G(x,y)=\frac{1}{|x-y|^{N-2}}-H(x,y)$, for $x,y\in\Om$, where $H$ is the regular part. These functions are symmetric: $G(x,y)=G(y,x)$ and $H(x,y)=H(y,x)$. 

Now we state our main result about the existence of nodal solutions that are towers of bubbles concentrating at the origin.

\begin{theorem}\label{theorem-tower of bubble}
Let $\mu=\mu_0\eps$ with $\mu_0>0$ fixed. For any given integer $k\ge0$ there exists $\eps_0>0$ such that for any $\eps\in(0,\eps_0)$, there exists a pair of solutions $\pm u_\eps$ to problem \eqref{pro} satisfying 
\begin{equation*}
u_\eps(x)
 = C_\mu(-1)^k\left(\frac{\si^\eps}{(\si^\eps)^2|x|^{\be_1}
     + |x|^{\be_2}}\right)^{\frac{N-2}{2}}
   + C_0\sum_{i=1}^{k}(-1)^{i-1}\left(\frac{\de_i^\eps}{(\de_i^\eps)^2
     +|x-\xi_i^\eps|^2}\right)^{\frac{N-2}{2}}
   + o(1)\ \ \text{as}\  \eps\to0.
	\end{equation*}
Here the constants $\si^\eps,\de_i^\eps>0$ and $\xi_i^\eps\in\R^N$ are determined as follows. There exist $\la_i^\eps,\ov{\la}^\eps>0$ and $\ze_i^\eps\in\R^N$, $i=1,\ldots,k$, with $\eta<\la_i^\eps,\ov{\la}^\eps<\frac1\eta$ and $|\ze_i^\eps|\le\frac{1}{\eta}$ for some $\eta>0$ small, so that: $\si^\eps = \ov{\la}^\eps\eps^{\frac{2(k+1)-1}{N-2}}$, $\de_i^\eps = \la_i^\eps\eps^{\frac{2i-1}{N-2}}$, $\xi_i^\eps = \de_i^\eps\ze_i^\eps$ for $i=1,2,\dots,k$.
\end{theorem}

The paper is organized as follows. In Section~2, we collect some notations and preliminary results. Section~3 is devoted to the proof of Theorem~\ref{theorem-tower of bubble}. Some useful technical lemmas are deferred to the appendices.

Throughout this paper, positive constants are denoted by $C, c$ and may vary from line to line. 

\section{\textbf{Notations and preliminary results}\label{Section 2}}

Here we collect some results from \cite{BarGuo-ANS}. Let $\iota_\mu^*:L^{2N/(N+2)}(\Om)\to H_\mu(\Om)$ be the adjoint operator of the inclusion $\iota_\mu:H_\mu(\Om)\to L^{2N/(N-2)}(\Om)$ as in \cite{FelliPis06CPDE}, that is,
\begin{equation*}
\iota_\mu^*(u) = v\qquad\Longleftrightarrow\qquad
 (v,\phi)=\int_{\Om}u(x)\phi(x)dx,\quad\text{for all }\phi\in H_\mu(\Om).
\end{equation*}
This is continuous, so there exists $c>0$ such that
\begin{equation}\label{ineq-adjoint operator}
\|\iota_\mu^*(u)\|_{\mu}\le c\|u\|_{2N/(N+2)}.
\end{equation}
Now the problem \eqref{pro} is equivalent to the fixed point problem
\begin{equation}\label{fixed point-pro}
u=\iota_\mu^*(f_\eps(u)), u\in H_\mu(\Om),
\end{equation}
where $f_\eps(s)=|s|^{2^*-2-\eps}s$.

Let $P:H^1(\R^N)\to H_0^1(\Om)$ be the projection defined by $\De Pu=\De u$ in $\Om$, $Pu=0$ on $\pa\Om$. We need the following two propositions and one remark from \cite{BarGuo-ANS}.

\begin{proposition}\label{proposition-eigenvalue}
Let $0<\mu<\overline{\mu}$ be fixed, and let $\La_i$, $i=1,2,\dots$, be the eigenvalues of
\begin{equation*}
\begin{cases}
-\De u-\mu\frac{u}{|x|^2} = \La|V_\si|^{2^{\ast}-2}u &\quad\text{in } \R^N,\\
|u|\to0 &\quad\text{as }|x|\to+\infty,
\end{cases}
\end{equation*}
in increasing order. Then $\La_1=1$ with eigenfunction $V_\si$ and $\La_2=2^*-1$ with eigenfunction $\frac{\pa V_\si}{\pa \si}$.
\end{proposition}

Setting $d_{\inf}:=\inf\{|x|:x\in\pa\Om\}$ and $d_{\sup}:=\sup\{|x|:x\in\pa\Om\}$ we have

\begin{proposition}\label{proposition-projection estimate}
Let $0<\mu<\overline{\mu}$ be fixed. Then for $\sigma>0$ the function $\vphi_\si:=V_\si-PV_\si$ satisfies
\begin{equation*}
0\le\vphi_\si\le V_\si \quad\text{and}\ \ 
\vphi_\si(x)
 = C_\mu(\ov{d}(x))^{\sqrt{\ov{\mu}}-\sqrt{\ov{\mu}-\mu}}H(0,x)\si^{\frac{N-2}{2}}+\hbar_\si(x);
\end{equation*}
with  
\begin{equation*}
d_{\inf}\le \ov{d}\le d_{\sup}\ \ \text{and} \ \ \hbar_\si=O(\si^{\frac{N+2}{2}}),\quad\frac{\pa \hbar_\si}{\pa \si}=O(\si^{\frac{N}{2}})~~\text{as}\ \sigma\to 0.
\end{equation*}
\end{proposition}

\begin{remark}
\begin{itemize}
\item[a)] If $\mu\to0^+,$ then
\begin{equation}\label{equ-1-projection estimate}
\vphi_\si(x) = C_0H(0,x)\si^{\frac{N-2}{2}}+O(\mu\si^{\frac{N-2}{2}})+\hbar_{\mu,\si}(x),
\end{equation}
where $\hbar_{\mu,\si}$ satisfies $\hbar_{\mu,\si}(x)=O(\si^{\frac{N+2}{2}}), \frac{\pa \hbar_{\mu,\si}(x)}{\pa \si}=O(\si^{\frac{N}{2}})~~\text{as}\ \sigma\to 0.$ 
\item[b)] Let us recall the similar results for $U_{\de,\xi}$ obtained in \cite{Rey90JFA}, that is
\begin{equation}\label{equ-3-projection estimate}
0 \le \vphi_{\de,\xi}:=U_{\de,\xi}-PU_{\de,\xi}\le U_{\de,\xi},~\vphi_{\de,\xi}
  = C_0 H(\xi,\cdot)\de^{\frac{N-2}{2}}+O(\de^{\frac{N+2}{2}}),
\end{equation}
as $\de\to 0$, uniformly in compact subsets of $\Omega$.
\end{itemize}
\end{remark}

\section{\textbf{Solutions with tower of bubbles concentrating at the origin}\label{Section 3}}

\subsection{\textbf{The finite-dimensional reduction}\label{Subsection 3.1}}

Let the integer $k\ge0$ be fixed. For $\eps>0$ small, $\la = (\la_1,\ldots,\la_k,\ov{\la})\in \R_+^{k+1}$ and $\ze = (\ze_1,\ldots,\ze_k)\in(\R^N)^k$ we set $\si:=\ov{\la}\eps^{\frac{2(k+1)-1}{N-2}}$, $\de_i := \la_i\eps^{\frac{2i-1}{N-2}}$,
\[
\xi = (\xi_1,\ldots,\xi_{k}):=(\de_1\ze_1,\ldots,\de_k\ze_{k}) \in \Om^k,
\]
and define
\[
W_{\eps,\la,\ze}
 :=\sum_{i=1}^{k}\Ker\left(-\De-(2^*-1)U_{\de_i,\xi_i}^{2^*-2}\right)
     + \Ker\left(-\De-\frac{\mu}{|x|^2}-(2^*-1)V_\si^{2^*-2}\right).
\]
By Proposition~\ref{proposition-eigenvalue} and \cite{BianEg91JFA} we know
\[
W_{\eps,\la,\ze}
 = \span\left\{\ov{\Psi},\ \Psi_i^0,\ \Psi_i^j:\ i=1,2,\dots,k,\ j=1,2,\dots,N\right\},
\]
where for $i=1,2,\dots,k$ and $j=1,2,\dots,N$:
\begin{equation}\label{eq:def-psi}
\Psi_i^j:=\frac{\pa U_{\de_i,\xi_i}}{\pa \xi_{i,j}},\quad
\Psi_i^0:=\frac{\pa U_{\de_i,\xi_i}}{\pa \de_i},\quad
\ov{\Psi}:=\frac{\pa V_{\si}}{\pa \si}
\end{equation}
with $\xi_{i,j}$ the $j-$th component of $\xi_i$.

We also need the spaces 
\[
K_{\eps,\la,\ze}:=P W_{\eps,\la,\ze},
\]
and
\[
K^\bot_{\eps,\la,\ze}
 :=\{\phi\in H_\mu(\Om):(\phi,P\Psi)=0,\text{ for all }\Psi\in W_{\eps,\la,\ze}\},
\]
as well as the $(\cdot,\cdot)_\mu$-orthogonal projections
\[
\Pi_{\eps,\la,\ze} : H_\mu(\Om)\to K_{\eps,\la,\ze},
\]
and
\[
\Pi^\bot_{\eps,\la,\ze} := Id-\Pi_{\eps,\la,\ze}:H_\mu(\Om)\to K^\bot_{\eps,\la,\ze}.
\]

For $\eps>0$ small we want to find solutions of \eqref{pro} close to
\begin{equation*}\label{shape of V-2}
V_{\eps,\la,\ze}:=\sum_{i=1}^{k}(-1)^{i-1}PU_{\de_i,\xi_i}+(-1)^k PV_\si,
\end{equation*}
where 
\[
(\la,\ze)\in\cO_\eta
 := \left\{(\la,\ze)\in\R_+^{k+1}\times (\R^N)^{k}:
         \la_i\in(\eta,\eta^{-1}),\ \ov{\la}\in(\eta,\eta^{-1}),\ |\ze_i|\le\frac1\eta,\
         i=1,\dots,k\right\}
\]
for some $\eta\in(0,1)$. This is equivalent to finding  $\eta>0$, $(\la,\ze)\in\cO_\eta$ and $\phi_{\eps,\la,\ze}\in K^\bot_{\eps,\la,\ze}$ such that $V_{\eps,\la,\ze} + \phi_{\eps,\la,\ze}$ solves \eqref{fixed point-pro}, hence 
\begin{equation}
\label{main-equality Com-2}
\Pi^\bot_{\eps,\la,\ze}\left(V_{\eps,\la,\ze} + \phi_{\eps,\la,\ze}
      - \iota_\mu^*(f_\eps(V_{\eps,\la,\ze} + \phi_{\eps,\la,\ze}))\right)
 =0
\end{equation}
and 
\begin{equation*}
\Pi_{\eps,\la,\ze}\left(V_{\eps,\la,\ze} + \phi_{\eps,\la,\ze}
      - \iota_\mu^*(f_\eps(V_{\eps,\la,\ze} + \phi_{\eps,\la,\ze}))\right)
 = 0.
\end{equation*}

Now we solve \eqref{main-equality Com-2} first for $\phi_{\eps,\la,\ze}$. Let us introduce the operator $L_{\eps,\la,\ze}:K^\bot_{\eps,\la,\ze}\to K^\bot_{\eps,\la,\ze}$ defined by
\[L_{\eps,\la,\ze}(\phi)=\phi-\Pi^\bot_{\eps,\la,\ze}\iota_\mu^*(f'_0(V_{\eps,\la,\ze})\phi).
\]
Take $\rho>0$ small enough and let
\begin{equation*}
A_{k+1}:=B(0,\sqrt{\de_{k+1}\de_k}),\quad A_i:=B(0,\sqrt{\de_i\de_{i-1}})\setminus B(0,\sqrt{\de_i\de_{i+1}})\ \text{ for } i=1,\ldots,k;
\end{equation*}
here $\de_0=\frac{\rho^2}{\de_1}$, $\de_{k+1}=\si$; cf.~\cite{MussoPis10JMPA}.

\begin{proposition}\label{proposition-operator L-2}
For any $\eta\in(0,1)$, there exist $\eps_0>0$ and $c>0$ such that for every $(\la,\ze)\in\cO_\eta$ and for every $\eps\in(0,\eps_0)$,
\begin{equation*}
\|L_{\eps,\la,\ze}(\phi)\|_{\mu}\ge c\|\phi\|_\mu
\quad \text{for all } \phi\in K^\bot_{\eps,\la,\ze}.
\end{equation*}
In particular, $L_{\eps,\la,\ze}$ is invertible with continuous inverse.
\end{proposition}

\begin{proof}
Following the same line as in \cite{MussoPis02IUMJ} we argue by contradiction. Suppose that there exist $\eta>0$, sequences
$\eps^n>0$, $(\la^n,\ze^n)\in\cO_\eta$, $\phi^n\in H_\mu(\Om)$ satisfying
$$\eps^n\to0,~ \la_i^n\to\la_i,~ \ov{\la}^n\to\ov{\la},~ \ze_i^n\to\ze_i,$$ as $n\to\infty$, and such that
\begin{equation*}
\phi^n\in K^\bot_{\eps^n,\la^n,\ze^n},\quad \|\phi^n\|_\mu=1,
\end{equation*}
and
\begin{equation}\label{equ-2-proposition-operator L-2}
L_{\eps^n,\la^n,\ze^n}(\phi^n)=h^n\ \text{ with } \|h^n\|_\mu\to 0;
\end{equation}
here $\la^n=(\la_1^{n},\ldots,\la_{k}^{n},\ov{\la}^n)$, $\ze^n=(\ze_1^{n},\ldots,\ze_{k}^{n})$, and for $\eps>0$ small: 
$\si^n=\ov{\la}^n\eps^{\frac{2(k+1)-1}{N-2}}$, 
$\xi^n = (\xi_1^n,\ldots,\xi_{k}^n) = (\de_1^n\ze_1^n,\de_2^n\ze_2^n,\dots,\de_k^n\ze_{k}^n) \in \Om^{k}$, 
$\de_i^n=\la_i^n\eps^{\frac{2i-1}{N-2}}$ for $i=1,2,\dots,k$. Consider the sets
\[
\textstyle
A_{k+1}^n:=B\left(0,\sqrt{\de_{k+1}^n\de_k^n}\right),\ \
A_i^n := B\left(0,\sqrt{\de_i^n\de_{i-1}^n}\right)\setminus
            B\left(0,\sqrt{\de_i^n\de_{i+1}^n}\right),\ i=1,2,\dots,k,
\]
where $\de_0^n:=\frac{\rho^2}{\de_1^n}$ and $\de_{k+1}^n:=\si^n$. Thus we have:
\begin{equation}\label{equ-3-proposition-operator L-2}
\phi^n-\iota_\mu^*\left(f'_0(V_{\eps^n,\la^n,\ze^n})\phi^n\right)
 = h^n-\Pi_{\eps^n,\la^n,\xi^n}\left(\iota_\mu^*(f'_0(V_{\eps^n,\la^n,\xi^n})\phi^n)\right).
\end{equation}
As in the proof of \cite[Proposition 4.1]{BarGuo-ANS}, we obtain
\[
w^n
 := -\Pi_{\eps^n,\la^n,\xi^n}(\iota_\mu^*(f'_0(V_{\eps^n,\la^n,\xi^n})\phi^n))
  = \sum_{i=1}^{k}\sum_{j=0}^N c_{i,j}^nP(\Psi_i^j)_n+c_0^n P(\ov{\Psi})_n
\]
for some coefficients $c_{i,j}^n$, $c_0^n$, where $(\Psi_i^j)_n$, $j=1,\ldots,N,(\Psi_i^0)_n$, and $(\ov{\Psi})_n$ are defined analogously to \eqref{eq:def-psi}. We argue in three steps.

\emph{Step 1.} We claim that
\begin{equation}\label{equ-wn-proposition-operator L-2}
 \lim_{n\to\infty}\|w^n\|_\mu=0.
\end{equation}
Multiplying \eqref{equ-3-proposition-operator L-2} by
$\De P(\Psi_l^h)_n + \mu\frac{P(\Psi_l^h)_n}{|x|^2}$, using Lemma~\ref{e56-e62-Lemma B},
Lemma~\ref{e50-e51}, Lemma~\ref{e83-85-Lemma B}, and arguing as in the proof of \cite[Proposition 4.1]{BarGuo-ANS}, we deduce $c_{l,h}^n\to0$, for $l=1,\dots,k$, $h=0,1,\ldots,N$, and  $c_0^n\to0$, as $n\to\infty$.
Thus the claim $\lim\limits_{n\to\infty}\|w^n\|_\mu=0$ follows.

\emph{Step 2.} As in \cite{MussoPis10JMPA} we use cut-off functions $\chi_i^n$, $i=1,\dots,k+1$, with the properties
\[
\left\{
\begin{aligned}
&\chi_i^n(x)=1
 \quad \text{if }\sqrt{\de_i^n\de_{i+1}^n} \le |x|\le\sqrt{\de_i^n\de_{i-1}^n}; \\
&\chi_i^n(x)=0
 \quad \text{if }|x| \le \frac{\sqrt{\de_i^n\de_{i+1}^n}}{2}
        \text{ or } |x| \ge 2\sqrt{\de_i^n\de_{i-1}^n};\\
&|\nabla\chi_i^n(x)| \le \frac{1}{\sqrt{\de_i^n\de_{i-1}^n}}
\quad\text{and} |\nabla^2\chi_i^n(x)| \le \frac{4}{\de_i^n\de_{i-1}^n},
\end{aligned}
\right.
\]
for $i=1,\dots,k$, and
\[
\left\{
\begin{aligned}
&\chi_{k+1}^n(x)=1, \quad \text{if } |x|\le\sqrt{\de_{k+1}^n\de_{k}^n}; \\
&\chi_{k+1}^n(x)=0, \quad \text{if } |x|\ge2\sqrt{\de_{k+1}^n\de_{k}^n};\\
&|\nabla\chi_{k+1}^n(x)|\le\frac{1}{\sqrt{\de_{k+1}^n\de_{k}^n}},
\quad\text{and}\quad
|\nabla^2\chi_{k+1}^n(x)|\le \frac{4}{\de_{k+1}^n\de_{k}^n}.
\end{aligned}
\right.
\]
The functions $\phi_i^n$ defined by
\[
\phi_i^n(y):=(\de_i^n)^{\frac{N-2}{2}}\phi^n(\de_i^n y)\chi_i^n(\de_i^n y),
\quad\text{for } y\in \Om_i^n:=\frac{\Om}{\de_i^n},\ i=1,\dots, k+1.
\]
are bounded in $D^{1,2}(\R^N)$. Therefore we may assume, up to a subsequence,
\[
\phi_i^n \weakto \phi_i^\infty\
\text{ weakly in } D^{1,2}(\R^N),\ i=1,2,\dots,k+1.
\]
Now we prove
\begin{equation}\label{equ-6-proposition-operator L-2}
\phi_i^\infty=0\quad\text{for } i=1,\dots,k+1.
\end{equation}
Again as in the proof of Proposition 4.1 in \cite{BarGuo-ANS}, using
\eqref{equ-3-proposition-operator L-2}, \eqref{equ-2-proposition-operator L-2} and 
\eqref{equ-wn-proposition-operator L-2}, we obtain for any $\psi\in C_0^\infty(\R^N)$ and $i=1,\dots,k$:
\[
\begin{aligned}
\int_{\Om_i^n}\nabla\phi_i^n(y)\nabla\psi(y)
 &= (\de_i^n)^{\frac{2-N}{2}}\int_{\Om}
       \nabla\iota_\mu^*\left(f'_0(V_{\eps^n,\la^n,\ze^n}(x))\phi^n(x)\right)
       \nabla\left(\chi_i^n(x)\psi\left(\frac{x}{\de_i^n}\right)\right)
     + o(1)\\
 &= (\de_i^n)^{\frac{2-N}{2}}\int_{\Om}
      f'_0\left(V_{\eps^n,\la^n,\ze^n}(x)\right)
      \phi^n(x)\chi_i^n(x)\psi\left(\frac{x}{\de_i^n}\right)
      + o(1)\\
 &= (\de_i^n)^{2}\int_{\Om_i^n}
      f'_0\left(V_{\eps^n,\la^n,\ze^n}(\de_i^n y)\right) \phi_i^n(y)\psi(y) + o(1)\\
 &= \int_{\R^N}f'_0\left(U_{1,\ze_i}(y)\right) \phi_i^\infty(y)\psi(y) + o(1).
\end{aligned}
\]
Hence $\phi_i^\infty$ is a weak solution to
\[
-\De\phi_i^\infty=f'_0(U_{1,\ze_i})\phi_i^\infty,~
\text{in}~D^{1,2}(\R^N).
\]
Setting $\Psi_{1,\ze_i}^j:=\frac{\pa U_{1,\ze_i}}{\pa \ze_{i,j}}$, for $j=1,\dots,N$, and $\Psi_{1,\ze_i}^0:=\frac{\pa U_{\de,\ze_i}}{\pa \de}|_{\de=1}$, we deduce as in \cite[Lemma 3.1]{MussoPis10JMPA}:
\[
\int_{\R^N}\nabla\phi_i^\infty(x)\nabla\Psi_{1,\ze_i}^j(x)=0,\qquad
j=0,1,2,\dots, N,~ i=1,2,\dots,k.
\]
Consequently \eqref{equ-6-proposition-operator L-2} holds for $i=1,\ldots,k$. The proof of $\phi_{k+1}^\infty=0$ is similar.

\emph{Step 3.} A contradiction arises as in the proof of \cite[Proposition 4.1]{BarGuo-ANS} and
\cite{MussoPis02IUMJ}.
\end{proof}

\begin{proposition}\label{proposition-estimate of error-2}
For every $\eta\in(0,1)$, there exist $\eps_0>0$, $c_0>0$ such that for every $(\la,\ze)\in\cO_\eta$ and every $\eps\in(0,\eps_0),$ there exists a unique solution
$\phi_{\eps,\la,\ze}\in K^\bot_{\eps,\la,\ze}$ of equation \eqref{main-equality Com-2} satisfying
\begin{equation*}
\|\phi_{\eps,\la,\ze}\|_{\mu}\le
c_0(\eps^{\frac{N+2}{2(N-2)}}+\eps^{\frac{2k+3}{4}}).
\end{equation*}
Moreover, the map $\Phi_\eps:\cO_\eta\to K^\bot_{\eps,\la,\ze}$ defined by $\Phi_\eps(\la,\ze):=\phi_{\eps,\la,\ze}$ is of class $C^1$.
\end{proposition}

\begin{proof}  As in \cite{BarMiPis06CVPDE}, solving \eqref{main-equality Com-2} is equivalent to finding a fixed point of the operator $T_{\eps,\la,\ze}:K^\bot_{\eps,\la,\ze}\to K^\bot_{\eps,\la,\ze}$ defined by
\[
T_{\eps,\la,\ze}(\phi)=L^{-1}_{\eps,\la,\ze}\Pi^\bot_{\eps,\la,\ze}
(\iota_\mu^*(f_\eps(V_{\eps,\la,\ze}+\phi)-f'_0(V_{\eps,\la,\ze})\phi)-V_{\eps,\la,\ze}).
\]
Now we prove that $T_{\eps,\la,\ze}$ is a contraction mapping. As in the proof of \cite[Proposition 4.2]{BarGuo-ANS}, using Proposition~\ref{proposition-operator L-2}, \eqref{ineq-adjoint operator} and Lemma~\ref{(M17)-Lemma B} we have
\[
\begin{aligned}
\|T_{\eps,\la,\ze}(\phi)\|_\mu
 &\le C\|f_\eps(V_{\eps,\la,\ze}+\phi)-f_\eps(V_{\eps,\la,\ze})
          -f'_\eps(V_{\eps,\la,\ze})\phi\|_{2N/(N+2)}\\\nonumber
&\hspace{1cm}
   + C\|(f'_\eps(V_{\eps,\la,\ze})-f'_0(V_{\eps,\la,\ze}))\phi\|_{2N/(N+2)}\\\nonumber
&\hspace{1cm}
   + C\|f_\eps(V_{\eps,\la,\ze})-f_0(V_{\eps,\la,\ze})\|_{2N/(N+2)}\\\nonumber
&\hspace{1cm}
   + C\left\|f_0(V_{\eps,\la,\ze})
        -\left(\sum_{i=1}^{k}(-1)^{i-1}f_0(U_{\de_i,\xi_i})
           +(-1)^kf_0(V_\si)\right)\right\|_{2N/(N+2)}\\\nonumber
&\hspace{1cm}
   + \sum_{i=1}^{k}O(\mu\de_i)+O\left(\left(\mu\si^{\frac{N-2}{2}}\right)^{\frac12}\right).
\end{aligned}
\]
Using Lemma \ref{e53-e55-B} and observing that
\[
\|f_\eps(V_{\eps,\la,\ze}+\phi) - f_\eps(V_{\eps,\la,\ze}) -
   f'_\eps(V_{\eps,\la,\ze})\phi\|_{2N/(N+2)}
\le C\|\phi\|_\mu^{2^*-1},
\]
we have
\[
\begin{aligned}
\|T_{\eps,\la,\ze}(\phi)\|_\mu
 &\le C\|\phi\|_\mu^{2^*-1}+C\eps\|\phi\|_\mu
       + C\eps+O\left(\eps^{\frac{N+2}{2(N-2)}}\right)+\sum_{i=1}^k O(\mu\de_i)
       + O\left(\left(\mu\si^{\frac{N-2}{2}}\right)^{\frac12}\right)\\
 &= C\|\phi\|_\mu^{2^*-1} + C\eps\|\phi\|_\mu + O\left(\eps^{\frac{N+2}{2(N-2)}}\right)
     + O\left(\eps^{\frac{2k+3}{4}}\right).
\end{aligned}
\]
The remaining part of the argument is standard and will therefore be left to the reader.
\end{proof}

For $\la=(\la_1,\ldots,\la_{k},\ov{\la})$ and $\ze=(\ze_1,\ldots,\ze_{k})$ we now consider the reduced functional
\[
I_\eps(\la,\ze)=J_\eps(V_{\eps,\la,\ze}+\phi_{\eps,\la,\ze}).
\]

\begin{proposition}\label{proposition-reducement-2}
If $(\la,\ze)\in\cO_\eta$ is a critical point of $I_\eps$ then $V_{\eps,\la,\ze}+\phi_{\eps,\la,\ze}$ is a solution of problem \eqref{pro} for $\epsilon>0$ small.
\end{proposition}

\begin{proof}
We omit it since it is similar to the proof of \cite[Proposition 4.3]{BarGuo-ANS}.
\end{proof}
\subsection{\textbf{Proof of Theorem \ref{theorem-tower of bubble}}\label{Subsection 3.2}}
We assume $\mu=\mu_0 \epsilon$ with $\mu_0>0$ fixed, and use the following notations from the above subsection. For
$$\eps>0 \text{ small},\ \la = (\la_1,\ldots,\la_k,\ov{\la})\in\R_+^{k+1},\ \ze = (\ze_1,\ldots,\ze_k)\in(\R^n)^k,$$
we set
$$\si=\ov{\la}\eps^{\frac{2(k+1)-1}{N-2}},\ \de_i = \la_i\eps^{\frac{2i-1}{N-2}},\ \xi = (\xi_1,\ldots,\xi_{k})=(\de_1\ze_1,\ldots,\de_k\ze_{k}),\ i=1,\ldots,k.$$
For convenience, we denote $\la_{k+1}:=\ov{\la}$ in this subsection.

\begin{lemma}\label{lemma-expansion of J-2}
For $\eps\to0^+$, there holds
\begin{eqnarray}\label{equ-expansion of J-2}
I_\eps(\la,\ze)&=&a_1+a_2\eps-a_3\eps\ln
\eps+\psi(\la,\ze)\eps+o(\eps)
\end{eqnarray}
$C^1$-uniformly with respect to $(\la,\ze)$ in compact sets of $\cO_\eta$. The constants are given by $$a_1=\frac{k+1}{N}S_0^{\frac{N}{2}},~
a_2 = \frac{(k+1)}{2^*}\int_{\R^N} U_{1,0}^{2^*}\ln U_{1,0}
 -\frac{k+1}{(2^*)^2}S_0^{\frac{N}{2}}-\frac{1}{2}
S_0^{\frac{N-2}{2}}\ov{S}\mu_0, ~
a_3=\frac{(k+1)^2}{2\cdot2^*}\int_{\R^N} U_{1,0}^{2^*}.$$
The function $\psi$ is given by
\[
\psi(\la,\ze)
 = b_1\la_{1}^{N-2} +\sum_{i=1}^{k}b_2(\frac{\la_{i+1}}{\la_i})^{\frac{N-2}{2}}h_1(\ze_i)
    -\sum_{i=1}^{k}b_3h_2(\ze_i)-b_4\ln(\la_1\dots\la_{k+1})^{\frac{N-2}{2}}\\
\]
with $$b_1=\frac{1}{2}C_0\int_{\R^N}U_{1,0}^{2^*-1},~ b_2=C_0^{2^*},~
b_3=\frac{1}{2}C_0^2\mu_0,~ b_4=\frac{1}{2^*}\int_{\R^N}U_{1,0}^{2^*},$$ and
\[
h_1(\ze_i) = \int_{\R^N}\frac{1}{|y+\ze_i|^{N-2}(1+|y|^2)^{\frac{N+2}{2}}},\quad
h_2(\ze_i) = \int_{\R^N}\frac{1}{|y+\ze_i|^2(1+|y|^2)^{N-2}}.
\]
\end{lemma}

\begin{proof} Observe that 
\begin{equation*}
\begin{aligned}
J_\eps(V_{\eps,\la,\ze})
&=\frac{1}{2}\int_{\Om}\left(|\nabla V_{\eps,\la,\ze}|^2
    -\mu\frac{|V_{\eps,\la,\ze}|^2}{|x|^2}\right)
  -\frac{1}{2^*}\int_{\Om}|V_{\eps,\la,\ze}|^{2^*}\\
&\hspace{1cm} +\left(\frac{1}{2^*}\int_{\Om}|V_{\eps,\la,\ze}|^{2^*}
-\frac{1}{2^*-\eps}\int_{\Om}|V_{\eps,\la,\ze}|^{2^*-\eps}\right)\\
&=(\Rmnum1)+(\Rmnum2)+(\Rmnum3).
\end{aligned}
\end{equation*}
For $k\ge1$, Lemma~\ref{e72-e78} and Lemma~\ref{e44-e48} yield:
\begin{equation}
\begin{aligned}
(\Rmnum1)&=\frac{1}{2}(k+1) S_0^{\frac{N}{2}}-\frac{N}{4} S_0^{\frac{N-2}{2}}\ov{S}\mu_0\eps
    -\frac12 C_0^{2^*}H(0,0)\la_1^{N-2}\int_{\R^N}\frac{1}{(1+|z|^2)^{\frac{N+2}{2}}}
       \cdot\eps\\
&\hspace{1cm} -\frac12\sum_{i=1}^{k}\mu_0 C_0^2\int_{\R^N}\frac{1}{|y|^2(1+|y-\ze_i|^2)^{N-2}}\cdot\eps\\
&\hspace{1cm}
 -C_0^{2^*}\left(\frac{\la_{k+1}}{\la_k}\right)^{\frac{N-2}{2}}
   \int_{\R^N}\frac{1}{(1+|y|^2)^{\frac{N+2}{2}}}\cdot\frac{1}{(1+|\ze_k|^2)^{\frac{N-2}{2}}}
     \cdot\eps\\\label{equ1-1-lemma-expansion of J-2}
&\hspace{1cm} -\sum_{i=1}^{k-1}C_0^{2^*}\left(\frac{\la_{i+1}}{\la_i}\right)^{\frac{N-2}{2}}
    \int_{\R^N}\frac{1}{(1+|y|^2)^{\frac{N+2}{2}}}\cdot\frac{1}{(1+|\ze_i|^2)^{\frac{N-2}{2}}}
        \cdot\eps + o(\eps).
\end{aligned}
\end{equation}
By Lemma \ref{e79} and Lemma \ref{e44-e48}, we obtain:
\begin{equation}
\begin{aligned}
(\Rmnum2)&=-\frac{1}{2^*}(k+1) S_0^{\frac{N}{2}} + \frac{N-2}{4} S_0^{\frac{N-2}{2}}\ov{S}\mu_0\eps
   + C_0^{2^*}H(0,0)\la_1^{N-2}\int_{\R^N} \frac{1}{(1+|z|^2)^{\frac{N+2}{2}}}\cdot\eps\\
&\hspace{1cm} +C_0^{2^*}\sum_{i=1}^{k}\left(\frac{\la_{i+1}}{\la_i}\right)^{\frac{N-2}{2}}
    \int_{\R^N}\frac{1}{|y|^{N-2}(1+|y-\ze_i|^2)^{\frac{N+2}{2}}}\cdot\eps\\\label{equ1-2-lemma-expansion of J-2}
&\hspace{1cm}
 +C_0^{2^*}\sum_{i=1}^{k}(\frac{\la_{i+1}}{\la_i})^{\frac{N-2}{2}}
   \int_{\R^N}\frac{1}{(1+|y|^2)^{\frac{N+2}{2}}}\frac{1}{(1+|\ze_i|^2)^{\frac{N-2}{2}}}
     \cdot\eps + o(\eps).
\end{aligned}
\end{equation}
Finally, Lemma \ref{e82} and Lemma \ref{e44-e48} imply:
\begin{equation}\label{equ1-3-lemma-expansion of J-2}
\begin{aligned}
(\Rmnum3)
&= -\frac{\eps}{(2^*)^2}(k+1)S_0^{\frac{N}{2}}
   - \frac{(N-2)\eps}{2\cdot2^*}\int_{\R^N} U_{1,0}^{2^*}\cdot\ln(\de_1\dots\de_k\si)\\
&\hspace{1cm} + \frac{(k+1)\eps}{2^*}\int_{\R^N} U_{1,0}^{2^*}\ln U_{1,0}+o(\eps)\\
&= - \frac{\eps}{(2^*)^2}(k+1)S_0^{\frac{N}{2}}
   - \frac{(N-2)\eps}{2\cdot2^*}
     \int_{\R^N} U_{1,0}^{2^*}\cdot\ln(\la_1\dots\la_k\ov{\la})\\
&\hspace{1cm}
   - \frac{(k+1)^2}{2\cdot2^*}\int_{\R^N} U_{1,0}^{2^*}\cdot\eps\ln\eps
   + \frac{(k+1)}{2^*}\int_{\R^N} U_{1,0}^{2^*}\ln U_{1,0}\cdot\eps + o(\eps).
\end{aligned}
\end{equation}
On the other hand, we deduce from Proposition \ref{proposition-estimate of error-2},  \eqref{equ-1-projection estimate},  \eqref{equ-3-projection estimate}, and Lemma~\ref{e53-e55-B} that
\begin{eqnarray}\label{equ1-4-lemma-expansion of J-2}
J_\eps(V_{\eps,\la,\ze}+\phi_{\eps,\la,\ze})-J_\eps(V_{\eps,\la,\ze})=o(\eps).
\end{eqnarray}
Now for $k\ge1$, \eqref{equ1-1-lemma-expansion of J-2}, \eqref{equ1-2-lemma-expansion of J-2},  \eqref{equ1-3-lemma-expansion of J-2} and \eqref{equ1-4-lemma-expansion of J-2} imply \eqref{equ-expansion of J-2}.

The case $k=0$ can be easily dealt with by using Lemma \ref{e12,e24,e30,e34}, Lemma \ref{e4,e41} and \eqref{equ1-4-lemma-expansion of J-2}. That \eqref{equ-expansion of J-2} holds $C^1$-uniformly with respect to $(\la,\ze)$ in compact sets of $\cO_\eta$ can be seen as in  \cite[Lemma~7.1]{MussoPis10JMPA}. We omit the details here.
\end{proof}


As a corollary of Lemma \ref{lemma-expansion of J-2} we obtain the following.

\begin{corollary}\label{Cor-expansion}
If $(\lambda,\zeta)$ is a stable (e.g.\ non-degenerate) critical point of $\psi(\lambda,\zeta)$, then $I_\epsilon$ has for $\epsilon>0$ small a critical point $(\lambda_\epsilon,\zeta_\epsilon)$ that converges towards $(\lambda,\zeta)$ as $\epsilon\to 0$.
\end{corollary}
\begin{proof}
The proof is standard.
\end{proof}

\begin{altproof}{\ref{theorem-tower of bubble}}
By the change of variables
\[
\la_1^{\frac{N-2}{2}}
 = s_1,\quad \left(\frac{\la_2}{\la_1}\right)^{\frac{N-2}{2}}=s_2,\ \ldots,\
    \left(\frac{\la_{k+1}}{\la_k}\right)^{\frac{N-2}{2}}=s_{k+1},
\]
$\psi(\la,\ze)$ can be rewritten as
\[
\wh{\psi}(s,\ze)
 = b_1 s_{1}^2+\sum_{i=1}^{k}b_2 s_{i+1} h_1(\ze_i)-\sum_{i=1}^{k}b_3 h_2(\ze_i)
    - b_4\ln(s_1^{k+1}s_2^{k}\dots s_{k+1}),
\]
where $s=(s_1,s_2,\ldots,s_{k+1})$.

For fixed $\ze$ the equation $\nabla_s\wh{\psi}(s,\ze)=0$ has the unique solution $\wh{s}(\ze)=(\wh{s}_1(\ze),\ldots,\wh{s}_{k+1}(\ze))$  with
\[
\wh{s}_1 = \sqrt{\frac{(k+1)b_4}{2b_1}},\quad
\wh{s}_2=\frac{k b_4}{b_2h_1(\ze_1)},\ \dots,\ \wh{s}_{k+1}=\frac{b_4}{b_2h_1(\ze_{k})}.
\]
It is easy to see that $\wh{s}(\ze)$ is non-degenerate. Plugging it into $\wh{\psi}(s,\ze)$ gives
\begin{equation}
\begin{aligned}
\wh{\psi}(\wh{s}(\ze),\ze)
 &= \frac{(k+1)^2b_4}{2}-\sum_{i=1}^{k}b_3 h_2(\ze_i)
    -b_4(\frac{k+1}{2}\ln\frac{(k+1)b_4}{2b_1}\\\nonumber
 &\hspace{1cm}
    + \sum_{i=1}^k i\ln\frac{i b_4}{b_2})+\sum_{i=1}^k b_4(k+1-i)\ln h_1(\ze_{i})\\
 &= C_1+\sum_{i=1}^k g_i(\ze_i),
\end{aligned}
\end{equation}
where
\[
C_1 = \frac{(k+1)^2b_4}{2}
   - b_4\left(\frac{k+1}{2}\ln\frac{(k+1)b_4}{2b_1}+\sum_{i=1}^k i\ln\frac{i b_4}{b_2}\right),
\]
and
\[
g_i(\ze_i)
 = b_4(k+1-i)\ln \int_{\R^N}\frac{1}{|y+\ze_i|^{N-2}(1+|y|^2)^{\frac{N+2}{2}}}
   - b_3 \int_{\R^N}\frac{1}{|y+\ze_i|^2(1+|y|^2)^{N-2}}.
\]
A direct computation shows that $\ze_i=0$ is a critical point of $g_i(\ze_i)$ such that
\[
\frac{\pa^2 g_i(\ze_i)}{\pa \ze_{i,j}\pa\ze_{i,l}}\Big|_{\ze_i=0} = 0 \quad \text{if } j\neq l;
\]
and
\[
\frac{\pa^2 g_i(\ze_i)}{\pa (\ze_{i,j})^2}\Big|_{\ze_i=0}
 = \frac{2N-8}{N}\int_{\R^N}\frac{b_3}{|y|^4(1+|y|^2)^{N-2}} > 0.
\]
Consequently $\ze_i=0$ is a nondegenerate local minimum of $g_i$. Hence $\ze=0$ is a $C^1$-stable critical point of $\wh{\psi}(\wh{s}(\ze),\ze)$. 
Thus we conclude by Corollary \ref{Cor-expansion} and Proposition \ref{proposition-reducement-2}.
\end{altproof}

\appendix
\section{Some lemmas from \cite{BarGuo-ANS}}\label{Appendix A}
In this part we collect some lemmas from \cite{BarGuo-ANS}. 
We define for $\eta\in(0,1)$:
\[
  \begin{aligned}
    \cT_\eta
      &:=\big\{(\la,\xi)\in\R_+^{k+1}\times\Om^{k}:\la_i\in(\eta,\eta^{-1}),\ov{\la}\in(\eta,\eta^{-1}),\
                 \dist(\xi_i,\pa\Om)>\eta,\\
      &\hspace{3cm}
       |\xi_i|>\eta,\ |\xi_{i_1}-\xi_{i_2}|>\eta,\ i,i_1,i_2=1,2,\dots,k,\ i_1\neq i_2\big\}.
  \end{aligned}
\]
%

\begin{lemma}\label{e50-e51}
(\rmnum1) For $i=1,2,\dots,k,$ and $j=0,1,\dots,N$, there holds
\begin{equation*}
\|P\Psi_i^j-\Psi_i^j\|_{2N/(N-2)}=\begin{cases}
O\left(\de_i^{\frac{N-2}{2}}\right)\quad& \text{if} ~~j=1,2,\dots,N,\\
O\left(\de_i^{\frac{N-4}{2}}\right)\quad& \text{if} ~~j=0\end{cases}
\end{equation*}
as  $\delta_i\to0$ uniformly for $\xi_i$ in a compact subset of $\Omega$.

\noindent(\rmnum2) There holds
\begin{equation*}
\|P\ov{\Psi}-\ov{\Psi}\|_{2N/(N-2)}=O\left(\si^{\frac{N-4}{2}}\right)
\end{equation*}
as  $\sigma\to0$, uniformly for $0<\mu<\overline{\mu}$.
\end{lemma}



\begin{lemma}\label{e12,e24,e30,e34}
For $i=1,2,\cdots,k,$ the following estimates hold uniformly for $(\la,\xi)\in\cT_\eta$:

\noindent (\rmnum1) For $\mu,\sigma\to0$:
\begin{equation*}
\begin{aligned}
&\int_{\Om}|\nabla
P V_\si|^2-\mu\frac{|P V_\si|^2}{|x|^2}\\
&\hspace{1cm}=
  S_\mu^{\frac{N}{2}}-C_0C_\mu^{2^*-1} H(0,0)\si^{N-2}
\int_{\R^N} \frac{1}{({|z|^{\be_1}}+|z|^{\be_2})^{\frac{N+2}{2}}}+O(\mu\si^{N-2})+O(\si^N).
\end{aligned}
\end{equation*}
\noindent (\rmnum2) For $\de_i,\sigma\to0$:
\begin{equation*}
\int_{\Om}|\nabla P U_{\de_i,\xi_i}|^2=S_0^{\frac{N}{2}}-C_0^{2^*}H(\xi_i,\xi_i)\de_i^{N-2}\int_{\R^N} \frac{1}{(1+|z|^2)^{\frac{N+2}{2}}}+o(\de_i^{N-2}).
\end{equation*}
\end{lemma}

\begin{lemma}\label{e4,e41}
For $\mu,\sigma\to0$ there holds
\begin{equation*}
\int_{\Om}|P V_\si|^{2^*}
=S_\mu^{\frac{N}{2}}
-2^*C_0C_\mu^{2^*-1}H(0,0)
\si^{N-2}
\int_{\R^N} \frac{1}{({|z|^{\be_1}}+|z|^{\be_2})^{\frac{N+2}{2}}}+O(\mu\si^{N-2})+O(\si^N).
\end{equation*}
\end{lemma}


\begin{lemma}\label{e44-e48}
For $\mu\to 0^+$ there holds
\begin{equation*}
\int_{\R^N}V_{1}^p=\int_{\R^N}U_{1,0}^p+o(1)\ \text{and}\  \int_{\R^N} V_{1}^p\ln
V_{1}=\int_{\R^N} U_{1,0}^p\ln
U_{1,0}+o(1)
\end{equation*}
for $p>1$ as well as 
\begin{equation*}
C_\mu=C_0-\frac{C_0}{N-2}\mu+O(\mu^2) \ \text{and}\
 S_\mu=S_0-\ov{S}\mu+O(\mu^2),
\end{equation*}
for some positive constant $\ov{S}$ independent of $\mu$.
\end{lemma}

\section{Proof of the lemmas from Section \ref{Section 3}}\label{Appendix B}


\begin{lemma}\label{e56-e62-Lemma B}
For $i,l=1,2,\dots,k,$ and $j,h=0,1,\dots,N$, with $i\neq l$ or $j\neq h$, there are constants $\wt{c}_0>0, \wt{c}_{i,j}>0$ such that the following estimates hold uniformly for $0<\mu<\ov{\mu}$:
\begin{equation*}
\begin{aligned}
(P\ov{\Psi},P\ov{\Psi})_\mu&=\wt{c}_0\frac{1}{\si^2}+o(\si^{-2})\ \text{as}\ \sigma\to0,\\
(P\ov{\Psi},P\Psi_i^j)_\mu&=o(\si^{-2})o(\de_i^{-2})\ \text{as}\ \sigma\to0, \ \de_i\to0,\ \text{uniformly~ for}\ \xi_i \ \text{in~a~compact~subset~of}\ \Omega,\\
(P\Psi_i^j,P\Psi_i^j)_\mu&=\wt{c}_{i,j}\frac{1}{\de_i^2}+o(\de_i^{-2})\ \text{as}\  \de_i\to0,\ \text{uniformly~ for}\ \xi_i \ \text{in~a~compact~subset~of}\ \Omega,\\
(P\Psi_i^j,P\Psi_l^h)_\mu&=o(\de_i^{-2})\ \text{as}\  \de_i\to0,\ \text{uniformly~for}\ \xi_i, \xi_l\ \text{in~a~compact~subset~of}\ \Omega.
\end{aligned}
\end{equation*}
\end{lemma}

\begin{proof}
We omit the proof since it is similar to \cite[Lemma A.1]{BarGuo-ANS}.
\end{proof}

\begin{lemma}\label{e83-85-Lemma B}
(\rmnum1) 
For $i,l=1,2,\cdots,k$ there holds
\begin{equation*}
\left\|\left(f'_0\left(\sum\limits_{i=1}^{k}(-1)^{i-1}PU_{\de_i,\xi_i}+(-1)^k PV_\si\right)-f'_0(U_{\de_l,\xi_l})\right)\Psi_l^h\right\|_{2N/(N+2)}=o\left(\de_l^{-\frac{2N}{N+2}}\right)
\end{equation*}
as $\sigma, \delta_i, \delta_l \to 0$ uniformly for $0<\mu<\ov{\mu}$ and $\xi_i$ in a compact subset of $\Omega$.

\noindent (\rmnum2) 
There holds
\begin{equation*}
\left\|\left(f'_0\left(\sum\limits_{i=1}^{k}(-1)^{i-1}PU_{\de_i,\xi_i}+(-1)^k PV_\si\right)-f'_0(V_\si)\right)\ov{\Psi}\right\|_{2N/(N+2)}
 = o\left(\si^{-\frac{2N}{N+2}}\right)
\end{equation*}
as $\sigma, \delta_i\to 0$ uniformly for $0<\mu<\ov{\mu}$ and $\xi_i$ in a compact subset of $\Omega$.
\end{lemma}

\begin{proof} We only prove (i) for $h\neq0$.
\begin{equation*}
\begin{aligned}
&\int_\Om |(f'_0(\sum\limits_{i=1}^{k}(-1)^{i-1}PU_{\de_i,\xi_i}+(-1)^k PV_\si)-f'_0(U_{\de_l,\xi_l}))\Psi_l^h|^{2N/(N+2)}\\\nonumber
&\hspace{1cm}=\bigcup\limits_{i=1}^{k+1}\int_{A_i} |(f'_0(\sum\limits_{i=1}^{k}(-1)^{i-1}PU_{\de_i,\xi_i}+(-1)^k PV_\si)-f'_0(U_{\de_l,\xi_l}))\Psi_l^h|^{2N/(N+2)}\\\nonumber
&\hspace{2cm}+\int_{\Om\backslash B(0,\rho)} |(f'_0(\sum\limits_{i=1}^{k}(-1)^{i-1}PU_{\de_i,\xi_i}+(-1)^k PV_\si)-f'_0(U_{\de_l,\xi_l}))\Psi_l^h|^{2N/(N+2)}.
\end{aligned}
\end{equation*}
As in \cite[Lemma A.3]{MussoPis10JMPA}, by \eqref{equ-1-projection estimate} and \eqref{equ-3-projection estimate} we have
\begin{equation}\label{equ2-e83-85-Lemma B}
\begin{aligned}
&\int_{A_l} |(f'_0(\sum\limits_{i=1}^{k}(-1)^{i-1}PU_{\de_i,\xi_i}+(-1)^k PV_\si)-f'_0(U_{\de_l,\xi_l}))\Psi_l^h|^{2N/(N+2)}\\
&\hspace{1cm}\le C\int_{A_l} |U_{\de_l,\xi_l}^{2^*-3}\vphi_{\de_l,\xi_l}\Psi_l^h|^{2N/(N+2)}+C\sum\limits_{i\neq l}\int_{A_l} |U_{\de_l,\xi_l}^{2^*-3}U_{\de_i,\xi_i}\Psi_l^h|^{2N/(N+2)}\\
&\hspace{2cm}+C\int_{A_l} |U_{\de_l,\xi_l}^{2^*-3}V_\si\Psi_l^h|^{2N/(N+2)}\\
&\hspace{1cm}\le o\left(\de_l^{-\frac{2N}{N+2}}\right),
\end{aligned}
\end{equation}
where we use
\[
\int_{A_l} |U_{\de_l,\xi_l}^{2^*-3}\vphi_{\de_l,\xi_l}\Psi_l^h|^{2N/(N+2)}
\le C\int_{A_l} |\frac{\de_l^{\frac{N+2}{2}}(x^h-\xi_l^h)}{(\de_l^2+|x-\xi_l|^2)^3}|^{2N/(N+2)}=O\left(\de_l^{\frac{2N(N-3)}{N+2}}\right),
\]
for $i\neq l$,
\begin{equation*}
\begin{aligned}
&\int_{A_l} |U_{\de_l,\xi_l}^{2^*-3}U_{\de_i,\xi_i}\Psi_l^h|^{2N/(N+2)}
=C\int_{A_l} |\frac{\de_l^2(x^h-\xi_l^h)}{(\de_l^2+|x-\xi_l|^2)^3}\frac{\de_i^{\frac{N-2}{2}}}{(\de_i^2+|x-\xi_i|^2)^{\frac{N-2}{2}}}|^{2N/(N+2)}\\
&\hspace{.5cm}
\le C(\int_{A_l} |\frac{\de_l^2(x^h-\xi_l^h)}{(\de_l^2+|x-\xi_l|^2)^3}|^{\frac{N}{2}})^{\frac{4}{N+2}}
(\int_{A_l}|\frac{\de_i^{\frac{N-2}{2}}}{(\de_i^2+|x-\xi_i|^2)^{\frac{N-2}{2}}}|^{2N/(N-2)})^{\frac{N-2}{N+2}}
=o\left(\de_l^{-\frac{2N}{N+2}}\right),
\end{aligned}
\end{equation*}
and similarly,
\begin{eqnarray*}
\int_{A_l} |U_{\de_l,\xi_l}^{2^*-3}V_\si\Psi_l^h|^{2N/(N+2)}=o\left(\de_l^{-\frac{2N}{N+2}}\right).
\end{eqnarray*}
The same arguments as for \eqref{equ2-e83-85-Lemma B} yield for $i\neq l$:
\begin{equation*}
\int_{A_i} \left|\left(f'_0(\sum\limits_{i=1}^{k}(-1)^{i-1}PU_{\de_i,\xi_i}+(-1)^k PV_\si)-f'_0(U_{\de_l,\xi_l})\right)\Psi_l^h\right|^{2N/(N+2)}
=o\left(\de_l^{-\frac{2N}{N+2}}\right).
\end{equation*}
Finally,
\begin{equation*}
\begin{aligned}
&\int_{\Om\backslash B(0,\rho)} \left|\left(f'_0(\sum\limits_{i=1}^{k}(-1)^{i-1}PU_{\de_i,\xi_i}+(-1)^k PV_\si)-f'_0(U_{\de_l,\xi_l})\right)\Psi_l^h\right|^{\frac{2N}{(N+2)}}\\
&\hspace{2cm}=\begin{cases}
O\left(\de_l^{\frac{N(N-2)}{N+2}}\right)\left(O\left(\si^{\frac{4N}{N+2}}\right)+\sum\limits_{i=1}^{k}O\left(\de_i^{\frac{4N}{N+2}}\right)\right)\quad& \text{if} ~~h=1,2,\dots,N,\\
O\left(\de_l^{\frac{N(N-4)}{N+2}}\right)\left(O\left(\si^{\frac{4N}{N+2}}\right)+\sum\limits_{i=1}^{k}O\left(\de_i^{\frac{4N}{N+2}}\right)\right)\quad& \text{if} ~~h=0.
\end{cases}
\end{aligned}
\end{equation*}
Then (i) follows.
\end{proof}

\begin{lemma}\label{(M17)-Lemma B}
There holds
\begin{equation*}
\left\|\iota_\mu^*\left(\sum\limits_{i=1}^{k}(-1)^{i-1}f_0(U_{\de_i,\xi_i})+(-1)^kf_0(V_\si)\right)-V_{\eps,\la,\ze}\right\|_\mu
 \le \sum\limits_{i=1}^{k}O(\mu\de_i)+O\left(\left(\mu\si^{\frac{N-2}{2}}\right)^{\frac{1}{2}}\right)
\end{equation*}
as $\mu,\sigma,\delta_i\to 0$  uniformly for $\xi_i$ in a compact subset of $\Omega$.
\end{lemma}

\begin{proof}
It is similar to \cite[Lemma A.4]{BarGuo-ANS}.
\end{proof}

\begin{lemma}\label{e53-e55-B}
For $\epsilon\to0$, the following estimates hold uniformly for $0<\mu<\ov{\mu}$ and $(\la,\ze)\in\cO_\eta$:
\begin{equation*}
\begin{aligned}
\|(f'_\eps(V_{\eps,\la,\ze})-f'_0(V_{\eps,\la,\ze}))\phi\|_{2N/(N+2)}
 &=O(\eps)\|\phi\|_\mu,\\
\|f_\eps(V_{\eps,\la,\ze})-f_0(V_{\eps,\la,\ze})\|_{2N/(N+2)}
 &=O(\eps),\\
\|f_0(V_{\eps,\la,\ze})-(\sum\limits_{i=1}^{k}(-1)^{i-1}f_0(U_{\de_i,\xi_i})+(-1)^k f_0(V_\si))\|_{2N/(N+2)}
 &= O\left(\eps^\frac{N+2}{2(N-2)}\right).
\end{aligned}\end{equation*}
\end{lemma}

\begin{proof}
The first two are from \cite{BarMiPis06CVPDE}. The last one can be proved as (4.5) in \cite{MussoPis10JMPA}.
\end{proof}

\begin{lemma}\label{e72-e78} Let $k\ge1$. Assume without loss of generality that $1\le i<j\le k$. Then the following estimates hold uniformly for $(\la,\ze)\in\cO_\eta$:
 
\noindent (\rmnum1) For $\epsilon\to0$:
\begin{equation*}
\int_{\Om}|\nabla
P V_\si|^2-\mu\frac{|P V_\si|^2}{|x|^2}=S_\mu^{\frac{N}{2}}+o(\eps).
\end{equation*}
\noindent (\rmnum2) For $\epsilon\to0$:
\begin{equation*}
\begin{aligned}
&\int_{\Om}\nabla
P V_\si\nabla PU_{\de_i,\xi_i}-\mu\frac{P V_\si PU_{\de_i,\xi_i}}{|x|^2}\\
&\hspace{1cm}=
\begin{cases}
C_0^{2^*}(\frac{\ov{\la}}{\la_k})^{\frac{N-2}{2}}\int_{\R^N}\frac{1}{(1+|y|^2)^{\frac{N+2}{2}}}\frac{1}{(1+|\ze_k|^2)^{\frac{N-2}{2}}}\cdot\eps+o(\eps)
\quad& \text{if} ~~i=k,\\
o(\eps)\quad& \text{if} ~~i\neq k.
\end{cases}
\end{aligned}
\end{equation*}
\noindent (\rmnum3) For $\epsilon\to0$:
\begin{equation*}
\mu\int_{\Om}\frac{|PU_{\de_i,\xi_i}|^2}{|x|^2}=\mu C_0^2\int_{\R^N}\frac{1}{|y|^2(1+|y-\ze_i|^2)^{N-2}}+o(\eps).
\end{equation*}
\noindent (\rmnum4) For $\epsilon\to0$:
\begin{equation*}
\mu\int_{\Om}\frac{PU_{\de_i,\xi_i}PU_{\de_j,\xi_j}}{|x|^2}
=o(\eps),~~~i\neq j.
\end{equation*}
\noindent (\rmnum5) For $\epsilon\to0$:
\begin{equation*}
\int_{\Om}|\nabla P U_{\de_i,\xi_i}|^2
 = \begin{cases}
S_0^{\frac{N}{2}}-C_0^{2^*}H(0,0)\la_1^{N-2}\int_{\R^N} \frac{1}{(1+|z|^2)^{\frac{N+2}{2}}}\cdot\eps+o(\eps)
\quad& \text{if} ~~i=1,\\
S_0^{\frac{N}{2}}+o(\eps)\quad& \text{if} ~~i\neq 1.
\end{cases}
\end{equation*}
\noindent (\rmnum6) For $\epsilon\to0$:
\begin{equation*}
\int_{\Om}\nabla P U_{\de_i,\xi_i}\nabla P U_{\de_j,\xi_j}\\
= \begin{cases}
C_0^{2^*}(\frac{\la_{i+1}}{\la_i})^{\frac{N-2}{2}}\int_{\R^N}\frac{1}{(1+|y|^2)^{\frac{N+2}{2}}}\frac{1}{(1+|\ze_i|^2)^{\frac{N-2}{2}}}\cdot\eps+o(\eps)
\quad& \text{if} ~~j=i+1,\\
o(\eps)\quad& \text{otherwise}.
\end{cases}
\end{equation*}
\end{lemma}

\begin{proof} (i) and (v) follow from Lemma \ref{e12,e24,e30,e34}.

Now we prove (ii). Using \eqref{equ-3-projection estimate}, integration by parts yields
\begin{equation}\label{equ7-Lemma B-e72-e78}
\int_{\Om}\nabla P V_\si\nabla PU_{\de_i,\xi_i}-\mu\frac{PV_\si PU_{\de_i,\xi_i}}{|x|^2}
=
\int_{\Om}V_\si^{2^*-1}(U_{\de_i,\xi_i}-\vphi_{\de_i,\xi_i})+\mu\int_\Om\frac{\vphi_\si (U_{\de_i,\xi_i}-\vphi_{\de_i,\xi_i})}{|x|^2}.
\end{equation}
It is easy to show, using \eqref{equ-1-projection estimate} and \eqref{equ-3-projection estimate}, that
\begin{equation}\label{equ8-Lemma B-e72-e78}
\begin{aligned}
\int_{\Om}V_\si^{2^*-1}\vphi_{\de_i,\xi_i}
 &\le C\de_i^{\frac{N-2}{2}}\int_{\Om}
\frac{\si^{\frac{N+2}{2}}}{(\si^2|x|^{\be_1}+|x|^{\be_2})^{\frac{N+2}{2}}}\\
 &\le C\si^{\frac{\ov{\mu}}{\sqrt{\ov{\mu}-\mu}}}\de_i^{\frac{N-2}{2}}\int_{\R^N}
\frac{1}{(|y|^{\be_1}+|y|^{\be_2})^{\frac{N+2}{2}}}=O(\si^{\frac{N-2}{2}}\de_i^{\frac{N-2}{2}}),
\end{aligned}
\end{equation}
and
\begin{equation}\label{equ9-Lemma B-e72-e78}
\mu\int_\Om\frac{\vphi_\si (U_{\de_i,\xi_i}-\vphi_{\de_i,\xi_i})}{|x|^2} \le \mu\int_\Om\frac{\vphi_\si U_{\de_i,\xi_i}}{|x|^2}\le O(\mu\si^{\frac{N-2}{2}}).
\end{equation}
On the other hand,
\begin{equation}\label{equ10-Lemma B-e72-e78}
\int_{\Om}V_\si^{2^*-1} U_{\de_i,\xi_i}=\bigcup\limits_{j=1}^{k+1}\int_{A_j}V_\si^{2^*-1} U_{\de_i,\xi_i}+O(\si^{\frac{N+2}{2}}\de_i^{\frac{N-2}{2}}).
\end{equation}
If $i=k,j=k+1$, then
\begin{equation}\label{equ11-Lemma B-e72-e78}
\begin{aligned}
\int_{A_{k+1}}V_\si^{2^*-1} U_{\de_k,\xi_k}
&= C_\mu^{2^*-1} C_0\si^{\frac{N+2}{2}}\de_k^{\frac{N-2}{2}}\int_{A_{k+1}}\frac{1}{({\si^2|x|^{\be_1}}
+|x|^{\be_2})^{\frac{N+2}{2}}}\frac{1}{(\de_k^2+|x-\xi_k|^2)^{\frac{N-2}{2}}}\\
&= C_\mu^{2^*-1} C_0\frac{\si^{\frac{\ov{\mu}}{\sqrt{\ov{\mu}-\mu}}}}{\de_k^{\frac{N-2}{2}}}
\int_{\frac{A_{k+1}}{\si^{\frac{\sqrt{\ov{\mu}}}{\sqrt{\ov{\mu}-\mu}}}}}\frac{1}{({|y|^{\be_1}}
+|y|^{\be_2})^{\frac{N+2}{2}}}\frac{1}{(1+|\frac{\si^{\frac{\sqrt{\ov{\mu}}}{\sqrt{\ov{\mu}-\mu}}}}{\de_k}y-\ze_k|^2)^{\frac{N-2}{2}}}\\
&= C_0^{2^*}(\frac{\ov{\la}}{\la_k})^{\frac{N-2}{2}}\int_{\R^N}\frac{1}{(1+|y|^2)^{\frac{N+2}{2}}}\frac{1}{(1+|\ze_k|^2)^{\frac{N-2}{2}}}\cdot\eps+o(\eps).
\end{aligned}
\end{equation}
If $i\neq k$ or $j\neq k+1$, then similar arguments as for \eqref{equ11-Lemma B-e72-e78} yield
\begin{equation}\label{equ12-Lemma B-e72-e78}
\int_{A_j}V_\si^{2^*-1} U_{\de_i,\xi_i}=o(\eps).
\end{equation}
Then we conclude by \eqref{equ7-Lemma B-e72-e78}-\eqref{equ12-Lemma B-e72-e78}.

Next (iii) follows from:
\begin{equation*}
\begin{aligned}
\mu\int_{\Om}\frac{|PU_{\de_i,\xi_i}|^2}{|x|^2}
 &=\mu\int_{\Om}\frac{|U_{\de_i,\xi_i}|^2}{|x|^2}+O(\mu\de_i^{N-2})\\
 &=\mu C_0^2\int_{\frac{\Om}{\de_i}}\frac{1}{|y|^2(1+|y-\ze_i|^2)^{N-2}}+O(\mu\de_i^{N-2})\\
 &=\mu
C_0^2\int_{\R^N}\frac{1}{|y|^2(1+|y-\ze_i|^2)^{N-2}}+o(\eps).
\end{aligned}
\end{equation*}

Similar arguments imply (iv).

For the proof of (vi), without loss of generality let $1\le i<j\le k$. Then 
\begin{equation}\label{equ13-Lemma B-e72-e78}
\begin{aligned}
&\int_{\Om}\nabla P U_{\de_i,\xi_i}\nabla P U_{\de_j,\xi_j}
 =\int_{\Om}U_{\de_j,\xi_j}^{2^*-1}U_{\de_i,\xi_i}+o(\eps)\\
 &\hspace{1cm}= \begin{cases}
C_0^{2^*}(\frac{\la_{i+1}}{\la_i})^{\frac{N-2}{2}}\int_{\R^N}\frac{1}{(1+|y|^2)^{\frac{N+2}{2}}}\frac{1}{(1+|\ze_i|^2)^{\frac{N-2}{2}}}\cdot\eps+o(\eps)
\quad& \text{if} ~~j=i+1,\\
o(\eps)\quad& \text{otherwise}.
\end{cases}
\end{aligned}
\end{equation}
\end{proof}

\begin{lemma}\label{e79}Let $k\ge1$. For $\epsilon\to0$ there holds, uniformly for $(\la,\ze)\in\cO_\eta$, setting $\la_{k+1}=\ov{\la}$: 
\begin{equation}\label{equ1-Lemma B-e79}
\begin{aligned}
&\int_{\Om}\left|\sum\limits_{i=1}^{k}(-1)^{i-1}PU_{\de_i,\xi_i}+(-1)^k
PV_\si\right|^{2^*}\\
 &\hspace{1cm}=
  kS_0^{\frac{N}{2}}+S_\mu^{\frac{N}{2}}-2^*C_0^{2^*}H(0,0)\la_1^{N-2}\int_{\R^N}
\frac{1}{(1+|z|^2)^{\frac{N+2}{2}}}\cdot\eps\\
 &\hspace{2cm}-2^*C_0^{2^*}\sum\limits_{i=1}^{k}(\frac{\la_{i+1}}{\la_i})^{\frac{N-2}{2}}\int_{\R^N}\frac{1}{|y|^{N-2}(1+|y-\ze_i|^2)^{\frac{N+2}{2}}}\cdot\eps\\\nonumber
 &\hspace{2cm}-2^*C_0^{2^*}\sum\limits_{i=1}^{k}(\frac{\la_{i+1}}{\la_i})^{\frac{N-2}{2}}
\int_{\R^N}\frac{1}{(1+|y|^2)^{\frac{N+2}{2}}}\frac{1}{(1+|\ze_i|^2)^{\frac{N-2}{2}}}\cdot\eps+o(\eps),
\end{aligned}
\end{equation}
\end{lemma}

\begin{proof} It is easy to see that
\begin{equation*}
\int_{\Om}\left|\sum\limits_{i=1}^{k}(-1)^{i-1}PU_{\de_i,\xi_i}+(-1)^k
PV_\si\right|^{2^*}
=\bigcup\limits_{j=1}^{k+1}\int_{A_j}|\sum\limits_{i=1}^{k}(-1)^{i-1}PU_{\de_i,\xi_i}+(-1)^k PV_\si|^{2^*}
+O(\de_1^N).
\end{equation*}
Observe that
\begin{equation*}
\begin{aligned}
\int_{A_k}V_\si U_{\de_k,\xi_k}^{2^*-1}
 &= C_\mu C_0^{2^*-1}\si^{\frac{N-2}{2}}\de_k^{\frac{N+2}{2}}\int_{A_k}\frac{1}{({\si^2|x|^{\be_1}}
+|x|^{\be_2})^{\frac{N-2}{2}}}\frac{1}{(\de_k^2+|x-\xi_k|^2)^{\frac{N+2}{2}}}\\
 &= C_\mu C_0^{2^*-1}\frac{\si^{\frac{N-2}{2}}}{\de_k^{\sqrt{\ov{\mu}-\mu}}}\int_{\frac{A_k}{\de_k}}\frac{1}{({(\frac{\si}{\de_k})^2|y|^{\be_1}}
+|y|^{\be_2})^{\frac{N-2}{2}}}\frac{1}{(1+|y-\ze_k|^2)^{\frac{N+2}{2}}}\\
 &= C_0^{2^*}(\frac{\ov{\la}}{\la_k})^{\frac{N-2}{2}}\int_{\R^N}\frac{1}{|y|^{N-2}(1+|y-\ze_k|^2)^{\frac{N+2}{2}}}\cdot\eps+o(\eps),
\end{aligned}
\end{equation*}
and
\begin{equation*}
\int_{A_j}V_\si U_{\de_i,\xi_i}^{2^*-1}=o(\eps), ~\text{if}~i\neq k,~\text{or}~j\neq k.
\end{equation*}
From \cite{MussoPis10JMPA}, for $1\le i<j\le k$ we also have
\begin{equation*}
\int_{A_l}U_{\de_j,\xi_j}U_{\de_i,\xi_i}^{2^*-1}+o(\eps)\\
 =\begin{cases}
C_0^{2^*}(\frac{\la_{i+1}}{\la_i})^{\frac{N-2}{2}}\int_{\R^N}\frac{1}{|y|^{N-2}(1+|y-\ze_i|^2)^{\frac{N+2}{2}}}\cdot\eps+o(\eps)
\quad& \text{if} ~~j=i+1,i=l,\\
o(\eps)\quad& \text{otherwise}.
\end{cases}
\end{equation*}
Using \eqref{equ11-Lemma B-e72-e78}, \eqref{equ13-Lemma B-e72-e78} and the above three equalities, the proof of \eqref{equ1-Lemma B-e79} is contained in Lemma 6.2 in \cite{MussoPis10JMPA}.
\end{proof}

\begin{lemma}\label{e82}For $\mu, \sigma, \delta_i \to 0$ there holds, uniformly in compact subsets of $\Omega$,
\begin{equation*}
\begin{aligned}
&\int_{\Om}|\sum\limits_{i=1}^{k}(-1)^{i-1}PU_{\de_i,\xi_i}+(-1)^k PV_\si|^{2^*}\ln|\sum\limits_{i=1}^{k}(-1)^{i-1}PU_{\de_i,\xi_i}+(-1)^k PV_\si|\\
&\hspace{1cm}=-\frac{N-2}{2}\ln\si\cdot\int_{\R^N} V_1^{2^*}-\frac{N-2}{2}\ln(\de_1 \de_2\dots\de_k)\cdot\int_{\R^N} U_{1,0}^{2^*}\\
&\hspace{2cm}+\int_{\R^N} V_1^{2^*}\ln V_1+k\int_{\R^N} U_{1,0}^{2^*}\ln U_{1,0}+o(1).
\end{aligned}
\end{equation*}
\end{lemma}

\begin{proof}
The proof is similar to the one of \cite[Lemma A.9]{BarGuo-ANS}.
\end{proof}

\noindent\textbf{Acknowledgements:} The authors would like to thank Professor Daomin Cao for many helpful discussions during the preparation of this paper. This work was carried out while Qianqiao Guo was visiting Justus-Liebig-Universit\"{a}t Gie{\ss}en, to which he would like to
express his gratitude for their warm hospitality.

\noindent\textbf{Funding:} Qianqiao Guo was supported  by the National Natural Science Foundation of China (Grant No. 11971385) and the Natural Science Basic Research Plan in Shaanxi Province of China (Grant No. 2019JM275).


\end{document}